\documentclass[11pt]{amsart}

\usepackage{etoolbox}
\usepackage{amsbsy}
\usepackage{amsmath}
\usepackage[alphabetic, backrefs, msc-links]{amsrefs}
\usepackage{amssymb}
\usepackage{amsthm}
\usepackage{amsxtra}
\usepackage{cancel}
\usepackage{comment}
\usepackage{fullpage}
\usepackage{graphicx}
\usepackage{hyperref}
\usepackage{paralist}
\usepackage{mathrsfs}
\usepackage{stmaryrd}
\usepackage{textcomp}
\usepackage[all,2cell]{xy}
\usepackage{enumitem}
\usepackage{pxfonts}
\usepackage{fancyhdr}
\usepackage{tikz}
\usepackage{tikz-qtree}
\usetikzlibrary{cd,matrix,arrows,decorations.pathmorphing}
\usepackage{bbm}
\makeatletter
\pgfqkeys{/tikz/cs}{
	latitude/.store in=\tikz@cs@latitude,
	longitude/.style={angle={#1}},
	theta/.style={latitude={#1}},
	rho/.style={angle={#1}}
}
\tikzdeclarecoordinatesystem{xyz spherical}{
	\pgfqkeys{/tikz/cs}{angle=0,radius=0,latitude=0,#1}%
	\pgfpointspherical{\tikz@cs@angle}{\tikz@cs@latitude}{\tikz@cs@xradius}
}
\makeatother

\tikzset{my color/.code=\pgfmathparse{(#1+90)/180*100}\pgfkeysalso{every path/.style={color=red!\pgfmathresult!blue}}}

\DeclareRobustCommand{\SkipTocEntry}[5]{}

\newtheorem{prop}{Proposition}[section]

\newtheorem{thm}[prop]{Theorem}
\newtheorem{cor}[prop]{Corollary}
\newtheorem{lem}[prop]{Lemma}

\theoremstyle{definition}
\newtheorem{defn}[prop]{Definition}

\newtheorem{cons}[prop]{Construction}

\newtheorem{notn}[prop]{Notation}

\theoremstyle{remark}
\newtheorem{cau}[prop]{Caution}
\newtheorem{conv}[prop]{Convention}

\newtheorem{rem}[prop]{Remark}

\renewcommand{\AA}{{\mathbb A}}
\newcommand{\bbf}{{\bf f}}
\newcommand{\bbg}{{\bf g}}

\newcommand{\DeclareMyOperator}[1]{%
	\expandafter\DeclareMathOperator\csname #1\endcsname{#1}
}
\newcommand{\DeclareMathOperators}{\forcsvlist{\DeclareMyOperator}}

\DeclareMathOperators{Spec, Sp, St, Str, Sch, Sets}
\DeclareMathOperators{UnSt, Un}

\numberwithin{equation}{prop}

\DeclareMathOperator{\Aff}{Aff}

\DeclareMathOperator{\Aut}{Aut}

\DeclareMathOperator{\B}{B}

\DeclareMathOperator{\Cart}{Cart}
\DeclareMathOperator{\Cat}{Cat}

\DeclareMathOperator{\Ch}{Ch}

\DeclareMathOperator{\colim}{colim}

\DeclareMathOperator{\DK}{DK}
\DeclareMathOperator{\End}{End}

\DeclareMathOperator{\Ext}{Ext}

\DeclareMathOperator{\Fun}{Fun}

\DeclareMathOperator{\GL}{GL}

\DeclareMathOperator{\Grp}{Grp}

\DeclareMathOperator{\Hom}{Hom}

\DeclareMathOperator{\id}{id}

\DeclareMathOperator{\Kan}{Kan}

\DeclareMathOperator{\Map}{Map}

\DeclareMathOperator{\Mod}{Mod}

\DeclareMathOperator{\Ob}{Ob}

\DeclareMathOperator{\Perf}{Perf}

\DeclareMathOperator{\PSt}{PSt}

\DeclareMathOperator{\RAut}{\mathbf RAut}

\DeclareMathOperator{\REnd}{\mathbf REnd}

\DeclareMathOperator{\RHom}{\mathbf RHom}

\DeclareMathOperator{\Set}{Set}

\DeclareMathOperator{\sk}{sk}

\DeclareMathOperator{\SMap}{\mathscr Map}

\DeclareMathOperator{\SREnd}{\mathscr End}

\DeclareMathOperator{\TwArr}{TwArr}

\DeclareMathOperator{\UHom}{\underline{Hom}}

\DeclareMathOperator{\WCat}{WCat}
\newcommand{\natmap}[5]{\xymatrix{
		#1 \ar@/^/[r]^{#3} \ar@/_/[r]_{#4} \ar@{}[r]|{\Downarrow_{#5}} & #2
}}
\newcommand{\adjoints}[4]{\xymatrix{#1 \ar@/^/[rr]^{#3} & \perp & #2 \ar@/^/[ll]^{#4}}}
\newcommand{\smalladjoints}[4]{\xymatrix@C=5pt{#1 \ar@/^/[rr]^{#3} & \perp & #2 \ar@/^/[ll]^{#4}}}
\newcommand{\xtworightarrows}[4]{\xymatrix{#1 \ar@<.5ex>[r]^{#3} \ar@<-.5ex>[r]_{#4} & #2}}
\newcommand{\xtwoleftarrows}[4]{\xymatrix{#1 &  \ar@<-.5ex>[l]_{#3} \ar@<.5ex>[l]^{#4}  #2}}
\newcommand{\xrightleftarrows}[4]{\xymatrix{#1 \ar@/^/[r]^{#3} & #2 \ar@/^/[l]^{#4}}}

\newcommand{\dg}[0]{\mathrm{dg}}

\newcommand{\op}[0]{\mathrm{op}}

\newcommand{\proj}[0]{\mathrm{proj}}

\newcommand{\qis}[0]{\mathrm{qis}}

\newcommand{\SHom}[1]{\mathscr Hom#1}

\usepackage{marginnote}

\theoremstyle{remark}

\theoremstyle{definition}

\theoremstyle{plain}

\newtheorem{corollary}[prop]{Corollary}

\newtheorem{theorem}[prop]{Theorem}
\newtheorem{lemma}[prop]{Lemma}

\newcommand{\msD}{{\mathscr D}}
\newcommand{\msX}{{\mathscr X}}

\newcommand{\msY}{{\mathscr Y}}

\newtheorem{innercustomthm}{Theorem}
\newenvironment{customthm}[1]
  {\renewcommand\theinnercustomthm{#1}\innercustomthm}
  {\endinnercustomthm}
  
\begin{document}

\title{Twisted forms of perfect complexes and Hilbert 90}

\author{Ajneet Dhillon}
\email{adhill3@uwo.ca}
\address{Western University, Canada}

\author{P\'al Zs\'amboki}
\email{zsamboki@renyi.hu}
\address{R\'enyi Institute, Hungary}

\begin{abstract}
	Automorphisms of a perfect complex naturally have the structure of an $\infty$-group: the 1-morphisms are
	quasi-isomorphisms, the 2-morphisms are homotopies, etc. This article starts by proving some basic properties of this $\infty$-group. We go on to study the deformation theory of this stack of $\infty$-groups and
	give a criterion for this stack to be formally smooth. The classifying stack of this $\infty$-group 
	classifies forms of a complex. We discuss a version of Hilbert 90 for perfect complexes.
 
\end{abstract}

\maketitle
\tableofcontents

\section{Introduction}

The purpose of this article is to study the automorphism group associated to a perfect complex $E$.
As perfect complexes live in derived categories, or some enhanced derived category, this object naturally acquires a higher categorical structure. Our purpose is to prove some elementary properties of ${\rm Aut}(E)$ such as this $\infty$-stack is in fact algebraic. We study its deformation theory.
The final section of this article give applications to Hilbert 90 type theorems.

There are a number of models for higher categories. In this article we will use quasi-categories.
In the the context of dg-categories a related result has appeared in  \cite{toen2007moduli}: let $S$ be an affine scheme, and $X\xrightarrow fS$ a smooth and proper morphism of schemes. Then the stack $\Perf_{X/S}$ of families of perfect complexes along $f$ is locally algebraic and locally of finite type \cite{toen2007moduli}*{Corollary 3.29}. Let $E$ be a perfect complex on $X$. This result implies that $\Aut_{X/S}(E)$ is algebraic. Using a similar argument to the proof in op.~cit., we  prove a similar
result but do not require that $f$ be smooth. More precisely,
\begin{customthm}{\ref{thm:RHom is algebraic}}
 Let $X\xrightarrow fS$ be a proper morphism of quasi-compact and quasi-separated schemes, and $E,F$ perfect complexes on $X$. Then the stack of families of morphisms $\SHom_{X/S}(E,F)$ is algebraic.
\end{customthm}
In particular, this shows that in this generality, $\Aut_{X/S}(E)$ is algebraic. This implies that the stack has
a presentation $P\rightarrow \Aut_{X/S}(E)$ where $P$ is a scheme, see \S 3.2. The presentation allows us to transport notions
from algebraic geometry to the study of this stack. For example, we say that the stack is smooth if
the scheme $P$ is smooth. For a more detailed discussion of how to lift central notions 
of algebraic geometry to higher stacks we refer the reader
\cite{gr}, \cite{HAG2} and \cite{stacksproject}. Note that terminology has not been settled in this subject.
What \cite{gr} call an Artin stack, is called a geometric stack in \cite{HAG2} and is referred to as
an algebraic stack in \cite{stacksproject}. 
We have decided to use the word ``algebraic'' in this work, see the start of \S 3.2 for a discussion of 
definitions. 
There are some differences in these cited references as to the 
underlying $\infty$-topos.

Section 2 of the paper starts with background information on $\infty$-categories and stacks. The main
object of study in this article, $\SHom_{X/S}(E,F)$ is
 introduced at the start of \S 3, see also
\ref{n:rshom}. After establishing that this presheaf of spaces is indeed
a stack, we proceed to equip it with a presentation and show it is algebraic. As a corollary we see that $\Aut_{X/S}(E)$ and its
classifying stack are algebraic, see \ref{cor:Aut is algebraic}
and \ref{cor:BAut is algebraic}.

In Section \ref{s:deformation theory} we study the deformation theory of the stack of automorphisms.
The deformation theory of complexes in derived categories has been studied in 
\cite{lieblich2006moduli} and \cite{huybrechts2010}. We modify the usual argument for deformations of a morphism of modules,
as presented in \cite{illusie}, to obtain obstructions and deformations.
The obstruction to lifting a point of ${\rm Aut}(E)$ lies inside $\Ext^1(E,E)$. In the case that
this obstruction vanishes, the $\infty$-group $\Aut(E)$ is smooth. Using this, we can prove a version of the Hilbert 90 Theorem for perfect complexes.

Let $X$ be a scheme, and $E,F$ be $\mathscr O_X$-modules. Let $\tau$ be a (Grothendieck) topology on $X$. We say that $F$ is a \emph{$\tau$-form of $E$}, if there exists a $\tau$-covering $U\to X$, and an isomorphism of $\mathscr O_U$-modules $E|U\cong F|U$. The classical Hilbert 90 theorem states that for a scheme $X$ and $n\ge1$, we have
$$
H^1_{\text{fppf}}(X,\GL_n)=H^1_{\text{Zar}}(X,\GL_n).
$$
That is, if $F$ is an fppf-form of $\mathscr O_X^{\oplus n}$, then it is actually a Zariski form: there exists a Zariski covering $V\to X$ and an isomorphism $F|V\cong\mathscr O_V^{\oplus n}$. In other words, $F$ is a locally free sheaf of rank $n$.

In general, the delooping $\B\Aut E$ classifies perfect complexes $F$ which are quasi-isomorphic to $E$ after a base change along an effective epimorphism $U\to X$. Since $\Aut E$ is algebraic, and thus fppf, the classifying map $X\xrightarrow{\mathbf c_E}\B\Aut E$ of $E$ is fppf, which implies that we can find an fppf covering $U\to X$ such that $E|U\simeq F|U$. Similarly, in case the $\infty$-group $\Aut E$ is smooth, then we can find a smooth covering $U\to X$ such that $E|U\simeq F|U$. This is explained in Proposition \ref{prop:BAut classifies forms} and Corollary \ref{cor:classifying map and forms}. Standard arguments now imply that any two forms must be quasi-isomorphic after some \'etale extension  (Proposition \ref{prop:etale Hilbert 90}). 

In Section \ref{s:Zariski Hilbert 90}, with a little more work, we can show that in fact they must be Zariski-locally quasi-isomorphic:
\begin{customthm}{\ref{thm:Zariski Hilbert 90}}[Hilbert 90 for perfect complexes]
 Let $X$ be a Noetherian scheme with infinite residue fields, and $E$ a perfect complex on $X$. Suppose that $F$ is an fppf-form of $E$, that is there exists an fppf covering $U\to X$ and a quasi-isomorphism $E|U\simeq F|U$. Then $F$ is a Zariski form of $E$, that is there exists a Zariski covering $V\to X$ and a quasi-isomorphism $E|V\simeq F|V$.
\end{customthm}
The argument for Zariski forms is independent of the deformation theory studied in Section \ref{s:deformation theory}.

\section*{Acknowledgments}
 We thank the referee for carefully reading our paper and making many useful comments. P\'al Zs\'amboki is partially supported by the project NKFIH K 119934. 
Ajneet Dhillon is partially supported by an NSERC grant.

\section{Background and Notation}

Let $S$ be a scheme. We denote by $\Sch_S$ the category of $S$-schemes which are quasi-compact and quasi-separated over $\Spec\mathbf Z$. We will equip this category with the fppf-topology. We denote by $\Aff_S$ the full subcategory of affine schemes with a structure morphism to $S$.

Our main object of study is the $\infty$-stack $\Perf_S$ of perfect complexes over $S$. In this section, we briefly explain the terminology we use for $\infty$-stacks, and a construction of $\Perf_S$.

When dealing with $\infty$-categorical structures, we will be using the point of view of quasi-categories. That is, an $\infty$-category is a quasi-category \cite{lurie2009higher}*{Definition 1.1.2.4}, in other words a fibrant object in the Joyal model structure on simplicial sets.

\begin{notn}\label{notn:Fun}
 Let $\Set_\Delta$ denote the category of simplicial sets.
If $x$ and $y$ are simplicial sets, we denote by $\Fun(x,y)$ the internal mapping space in $\Set_\Delta$. It is
the simplicial set whose $n$-simplices are given by
$$\Fun(x,y)_n = \Hom_{\Set_\Delta}(\Delta^n\times x, y).$$

\end{notn}

\begin{notn}

Let $x$ be a quasi-category. Then its opposite is denoted by $x^\op$. Given a morphism between quasi-categories $x\xrightarrow fy$, so that $f$ is a vertex 
of $\Fun(x,y)$,
its opposite is denoted by
$f^\op:x^\op \rightarrow y^\op$.
\end{notn}

\begin{defn}

Let $\Cat_\Delta$ denote the category of simplicial categories. Let
$$
\adjoints{\Set_\Delta}{\Cat_\Delta}{\mathfrak C}{N}
$$
denote the Quillen equivalence between the Joyal and Bergner model structures. Let $\Kan\in\Cat_\Delta$ denote the simplicial category with
\begin{enumerate}
 \item objects Kan complexes and
 \item for two Kan complexes $X,Y$ the mapping space $\Map_{\Kan}(X,Y)$ the internal mapping space $\Fun(X,Y)$.
\end{enumerate}
Then the \emph{quasi-category of spaces} is the simplicial nerve $\mathscr S:=N(\Kan)$.

\end{defn}

\begin{defn}
 Let $\Cat_\infty^\Delta$ denote the simplicial category with
\begin{itemize}
 \item objects quasi-categories, and
 \item the mapping space $\Map_{\Cat_\infty^\Delta}(C,D)$ the interior of the quasi-category that is the mapping simplicial set $\Fun(C,D)$.
\end{itemize}
Then the \emph{quasi-category of quasi-categories} is the simplicial nerve $\Cat_\infty:=N(\Cat_\infty^\Delta)$.
\end{defn}

\subsection{$\infty$-stacks}

\begin{defn}

Let $K$ be a quasi-category. In our case, this will be $N(\Sch_S)$ or $N(\Aff_S)$. From the abstract point of view, our objects of study are \emph{prestacks (presheaves of spaces) on $K$}, that is morphisms of quasi-categories $K^\op\to\mathscr S$. Via the straightening-unstraightening equivalence \cite{lurie2009higher}*{Theorem 2.2.1.2} these objects can be described as right fibrations over $K$. If a presheaf of spaces $K^\op\xrightarrow F\mathscr S$ corresponds to a right fibration $\mathscr X\xrightarrow pK$ via this equivalence, we say that \emph{$F$ is a classifying map for $p$}. The \emph{quasi-category of prestacks on $K$} is the presheaf quasi-category $\PSt(K)=\mathscr P(K)=\Fun(K^\op,\mathscr S)$.

\end{defn}

\begin{defn}

Note that if we want to consider mapping prestacks, we need to study presheaves of quasi-categories. There is a straightening-unstraightening result in this setting too \cite{lurie2009higher}*{Theorem 3.2.0.1}, thus we can view presheaves of quasi-categories on $K$ as Cartesian fibrations on $K$. Here too if a presheaf of quasi-categories $K^\op\xrightarrow F\Cat_\infty$ corresponds to a Cartesian fibration $\mathscr X\xrightarrow pK$ via this equivalence, we say that \emph{$F$ is a classifying map for $p$}.

\end{defn}

\begin{defn}\label{defn:interior}
 Let $\mathscr X\xrightarrow pK$ be a Cartesian fibration. Then its \emph{interior} $\mathscr X^\simeq\subseteq\mathscr X$ is the simplicial subset such that
 \begin{enumerate}
  \item the vertices are the same: $\mathscr X^\simeq_0=\mathscr X_0$,
  \item the edges $\mathscr X^\simeq_1\subseteq\mathscr X_1$ are exactly the $p$-Cartesian edges and
  \item for $n>1$, the $n$-simplices are the same: $\mathscr X^\simeq_n=\mathscr X_n$.
 \end{enumerate}
It is a right fibration \cite{lurie2009higher}*{Corollary 2.4.2.5}. Note that for $T\in K$ the simplicial subset of sections $\mathscr X^\simeq(T)\subseteq\mathscr X(T)$ is the largest $\infty$-subgroupoid of the $\infty$-category $\mathscr X(T)$.
\end{defn}

Now we will formulate descent in this setting.

\begin{defn}
 Let $\mathscr X\xrightarrow pK$ be an inner fibration. Then the \emph{quasi-category of sections} $\Gamma(K,\mathscr X)$
is the strict pullback
\begin{center}
	
	\begin{tikzpicture}[xscale=3,yscale=1.5]
	\node (C') at (0,1) {$\Gamma(K,\mathscr X)$};
	\node (D') at (1,1) {$\Fun(K,\mathscr X)$};
	\node (C) at (0,0) {$\{\id_K\}$};
	\node (D) at (1,0) {$\Fun(K,K)$.};
	\path[->,font=\scriptsize,>=angle 90]
	(C') edge (D')
	(C') edge (C)
	(D') edge node [right] {$p\circ$} (D)
	(C) edge (D);
	\end{tikzpicture}
	
\end{center}
We will say that a section $\sigma$, that is a $0$-simplex $\sigma\in \Gamma(K,\mathscr X)$,
is \emph{Cartesian} if $\sigma(e)$ is a $p$-Cartesian edge for each edge $e\in K_1$.

Let us denote by $\Gamma_{\Cart}(K,\mathscr X)\subset\Gamma(K,\mathscr X)$ the full $\infty$-subcategory on Cartesian sections.
\end{defn}

\begin{defn}

Let $\mathscr X\xrightarrow pK$ be a Cartesian fibration and $\Delta^\op_+\xrightarrow{k_+}K$ an augmented simplicical object in $K$. Let $k$ denote the restriction $k_+|\Delta^\op$. We say that \emph{$\mathscr X$ satisfies descent with respect to $k_+$}, if the restriction map
$$
\Gamma_{\Cart}(k_+,\mathscr X)\to\Gamma_{\Cart}(k,\mathscr X)
$$
is an equivalence.

\end{defn}

\begin{rem}
 This definition of descent is equivalent to the usual one: the Cartesian fibration $p$ has descent with respect to $k_+$ if and only if the restriction $\mathbf c_p|k_+$ of a classifying map $K^\op\xrightarrow{\mathbf c_p}\Cat_\infty$ is the limit diagram of the restriction $\mathbf c_p|k$ \cite{lurie2009higher}*{Proposition 3.3.3.1}.
\end{rem}

\begin{defn} 
Let $\tau$ be a Grothendieck topology on $K$ \cite{lurie2009higher}*{Definition 6.2.2.1}. Then we say that \emph{$\mathscr X$ satisfies $\tau$-descent}, if $\mathscr X$ satisfies descent with respect to the augmented \v Cech nerve of every $\tau$-covering.
 
\end{defn}

\begin{defn} Let $K^\op\xrightarrow F\mathscr S$ be a prestack on $K$. If it has $\tau$-descent, then we call it a \emph{$\tau$-stack}. The \emph{quasi-category of $\tau$-stacks} is the full $\infty$-subcategory $\St(K)\subseteq\PSt(K)$ on stacks. It is an $\infty$-topos \cite{lurie2009higher}*{Proposition 6.2.2.7}. In case $K=N(\Sch_S)$ we let $\St(S)=\St(N(\Sch_S))$.
\end{defn}

\subsection{The $\infty$-stack of perfect complexes}

Let $K=N(\Sch_S)$ and equip it with the fppf topology. Now we will describe a construction of the $\infty$-stack $\Perf_S$ of perfect complexes.

\begin{conv}
 In this paper we use the convention of cochain complexes. That is, for an abelian category $A$, we view a chain complex $M_\bullet$ in $A$ as the cochain complex $M^\bullet$ in $A$ by letting $M^n$ denote $M_{-n}$ and letting $d^n$ denote $d_{-n}$.
\end{conv}

\begin{rem}
 Here we discuss more formally the switch between the chain complex and the cochain complex conventions. The set $\mathbf Z$ of integers can be given both the natural ordering $\le$ and the reverse ordering $\ge$. Both these partially ordered sets induce categories: for integers $m,n\in\mathbf Z$, we have
 $$
 \Hom_{(\mathbf Z,\le)}(m,n)=\begin{cases}
                              * & m\le n\\
                              \emptyset & \text{else}
                             \end{cases}
\text{ and }\Hom_{(\mathbf Z,\ge)}(m,n)=\begin{cases}
                              * & m\ge n\\
                              \emptyset & \text{else.}
                             \end{cases}
 $$
 A chain complex in $A$ is a functor $(\mathbf Z,\le)^\op\to A$ and a cochain complex in $A$ is a functor $(\mathbf Z,\ge)^\op\to A$. To switch between the two, we can precompose with the opposite of the equivalence of categories 
 $$
 (\mathbf Z,\le)\xrightarrow{m\mapsto-m}(\mathbf Z,\ge).
 $$
\end{rem}

\begin{notn}

Let $D$ be a dg-category. Then we denote by $\sk_1D$ its underlying 1-category. 

Let $A$ be an abelian category. Then we denote by $\Ch(A)$ the dg-category of cochain complexes. That is, the objects of $\Ch(A)$ are the cochain complexes in $A$, and for two cochain complexes $M,N\in\Ch(A)$, the degree-$i$ part of the mapping complex $\UHom(M,N)$ is the collection of degree-$i$ morphisms $M\xrightarrow fN$, and for such a morphism, its differential $d_{\UHom(M,N)}f$ is $d_N\circ f-(-1)^if\circ d_M$.

Let $\Ch^-(A)\subseteq\Ch(A)$ denote the full dg-subcategory of bounded above cochain complexes in $A$. That is, we have $M\in\Ch^-(A)$ if and only if $M^i=0$ for $i\gg0$.

We denote by $A_\proj\subseteq A$ the full subcategory on projective objects.

\end{notn}

\begin{cons}

Let $\Spec R$ be an affine $S$-scheme. Then we can let the derived quasi-category be the localization $\mathscr D^-(\Spec R):=N(\sk_1\Ch^-(\Mod(R)))[\qis^{-1}]$ at the localizing class of quasi-isomorphisms. We can make the construction functorial by sending a morphism of affine $S$-schemes $\Spec R'\to\Spec R$ to the base change map $\Mod(R)\xrightarrow{\otimes_RR'}\Mod(R')$. This gives a morphism of quasi-categories $N(\Aff_S)^\op\to\WCat_\infty$ which we can compose with the functorial localization map $\WCat_\infty\xrightarrow{(\mathscr C,\qis)\mapsto C[\qis^{-1}]}\Cat_\infty$ \cite{lurie2014higher}*{Proposition 4.1.7.2}. 

We get a presheaf of quasi-categories
$$
N(\Aff_S)^\op\xrightarrow{T\mapsto\mathscr D^-(T)}\Cat_\infty.
$$
It satisfies fppf descent \cite{lurie2011descent}*{Theorem 7.5}.

To extend this to $S$-schemes, we can glue. That is, for an $S$-scheme $T$, we want
$$
\mathscr D^-(T)=\lim_{U\in N(\Aff_T)}\mathscr D^-(U).
$$
As we have $T=\colim_{U\in\Aff_T}U$, this can be done functorially in a unique up to homotopy way \cite{lurie2009higher}*{Theorem 5.1.5.6}. We let $\Perf_S\subseteq\mathscr D^-_S$ denote the full substack on perfect complexes.

\end{cons}

\begin{cau}
 Note that for $(\mathscr C,W)\in\Ob\WCat_\infty$ the localization $\mathscr C[W^{-1}]$ is a quasi-category and not a 1-category. We can get the 1-categorical localization as the homotopy category.
\end{cau}

\begin{rem}

A more concrete description of the localization $N(\Ch^-(\Mod(R)))[\qis^{-1}]$ can be obtained as a dg-nerve \cite{lurie2014higher}*{Construction 1.3.1.6}: since the category $\Mod(R)$ of $R$-modules has enough projectives, we get an equivalence of dg-categories \cite{lurie2014higher}*{Theorem 1.3.4.4}
$$
\mathscr D^-(\Spec R)=N(\Ch^-(\Mod(R)))[\qis^{-1}]\simeq N_\dg(\Ch^-(\Mod(R)_\proj)).
$$
Note in particular that the quasi-category $\Perf_S(\Spec R)$ of perfect complexes on $\Spec R$ is equivalent to the dg-nerve of the dg-category of strictly perfect complexes on $\Spec R$.

\end{rem}

\section{The stack of automorphisms of a complex.}

Fix a scheme $X$ and a  perfect complex $E$ on $X$. Pulling back along each map
\[
\Spec(R)\rightarrow X\]
we obtain a vertex in each $\msD^-(R)$ and this compatible family of vertices produces
a vertex in $\msD^-(X)$.

Let $E$, $F$ be two perfect complexes on $X$. Now we would like to describe the mapping stack $\SHom_X(E,F)$. Let $R$ be an $X$-algebra. Recall that by the Dold--Kan correspondence we have an equivalence
$$
N:\Mod(R)^{\Delta^\op}\leftrightarrows\Ch^{\le0}(\Mod R):\DK
$$
between the 1-category of simplicial $R$-modules and the 1-category of cochain complexes of $R$-modules vanishing in positive degrees. Using this, we can give an explicit description of the mapping spaces in the quasi-category $\mathscr D^-(R)$: we have \cite{lurie2014higher}*{Remark 1.3.1.12}
$$
\Map_{\mathscr D^-(R)}(E_R,F_R)\simeq\DK\tau_{\le0}\UHom(E_R,F_R)
$$
where $\UHom$ is the mapping cochain complex and $\tau_{\le0}$ is the cochain complex truncation map.

We will first give a natural construction of the mapping prestack between sections of a Cartesian fibration. Then we show that in case this Cartesian fibration is our object of study $\mathscr D^-$, the mapping spaces of sections over affine schemes can be explicitly described as above:
$$
\SHom_X(E,F)(\Spec R)\simeq\DK\tau_{\le0}\UHom(E_R,F_R).
$$


\subsection{Mapping spaces and automorphism groups in an $\infty$-topos}

Our original objects of study are forms of a perfect complex $E$ on an $S$-scheme $X$. Recall that if $E$ is a coherent sheaf, then forms of $E$ are classified by the gerbe $\B\Aut E$ where $\Aut E$ is the automorphism sheaf of $E$.

We can use the same approach in the higher categorical setting: The perfect complex $E$ is described by a classifying map $X\xrightarrow{\mathbf c_E}\Perf_X$. Just like the automorphism sheaf of a section of a 1-stack, the automorphism group of a section of an $\infty$-stack can be defined by the loop group construction:

\begin{defn}\label{d:loop} Let $\mathscr C$ be an $\infty$-topos. Let $*\xrightarrow cC$ be a pointed object in $\mathscr C$. Then the \emph{automorphism group $\Aut(c)_\bullet=\Aut_{C}(c)_\bullet$ of $c$} is the loop group $\Omega(c)_\bullet\in\Grp(\mathscr C)$ \cite{lurie2009higher}*{Lemma 7.2.2.11}. In particular, the underlying object $\Omega(c)=\Omega(c)_1$ fits in a homotopy pullback diagram
\begin{center}
	
	\begin{tikzpicture}[xscale=1.5,yscale=1.5]
	\node (C') at (0,1) {$\Omega(c)$};
	\node (D') at (1,1) {$*$};
	\node (C) at (0,0) {$*$};
	\node (D) at (1,0) {$\mathscr C.$};
	\path[->,font=\scriptsize,>=angle 90]
	(C') edge (D')
	(C') edge (C)
	(D') edge node [right] {$c$} (D)
	(C) edge node [above] {$c$} (D);
	\end{tikzpicture}
	
\end{center}
\end{defn}
$\infty$-groups in a quasi-category are described as simplicial objects $\Delta^\op\to\mathscr C$. Since $\mathscr C$ is an $\infty$-topos, the loop group $\Omega(c)_\bullet$ admits a homotopy colimit $\B\Omega(c)$. We will show that in case $\mathscr C=\St(S)$, $C$ is the interior $\Perf_X^\simeq\subseteq\Perf_X$ and $c=\mathbf c_E$, the delooping $\B\Aut(E)$ is an fppf stack. This shows that it classifies fppf forms of $E$.

For this, we will give another description of the underlying object $\Aut(E)$: it is the full substack of $\End(E)=\SHom(E,E)$ on self-equivalences $E\to E$. Using this and the local description using the Dold--Kan equivalence, we will be able to describe the local structure of $\B\Aut(E)$ explicitly.

We will start by defining mapping stacks in the language of Cartesian fibrations. We show that the mapping stack inherits descent properties from the original fibration.
We conclude the subsection by proving equivalent the two potential definitions of 
automorphism stack.

\begin{cons} Let $\mathscr C$ be a quasi-category with a final object $S$ and $\mathscr X\xrightarrow p\mathscr C$ a Cartesian fibration. Let $x,y\in\mathscr X(S)$ be two objects over $S$. By the quasi-categorical version of Quillen's Theorem A \cite{lurie2009higher}*{Theorem 4.1.3.1}, the inclusion map $\{S\}\to\mathscr C$ is cofinal, thus the restriction map
$$
\Gamma(\mathscr C,\mathscr X^\simeq)\to\Gamma(\{S\},\mathscr X^\simeq)=\mathscr X^\simeq(S)
$$
is a homotopy equivalence. Therefore, there exists a Cartesian section $\mathscr C\xrightarrow{y_\bullet}\mathscr X$ with $y_S=y$. Then the \emph{mapping prestack} can be formed as the strict fibre product
\begin{center}
	
	\begin{tikzpicture}[xscale=3,yscale=1.5]
	\node (C') at (0,1) {$\SMap(x,y)$};
	\node (D') at (1,1) {$\mathscr X^{/y}$};
	\node (C) at (0,0) {$\mathscr C$};
	\node (D) at (1,0) {$\mathscr X.$};
	\path[->,font=\scriptsize,>=angle 90]
	(C') edge (D')
	(C') edge (C)
	(D') edge (D)
	(C) edge node [above] {$y_\bullet$} (D);
	\end{tikzpicture}
	
\end{center}
 
\end{cons}

Let us now make our construction functorial in $(x,y)\in\mathscr X^\op(S)\times\mathscr X(S)$, that is construct a Cartesian fibration $\SMap\to\mathscr X(S)\times\mathscr C\times\mathscr X^\op(S)$ such that the pullback along the inclusion map $\{x\}\times\mathscr C\times\{y\}$ is $\SMap(x,y)$.

\begin{defn} To construct $\SMap$, we will make use of the twisted arrow construction \cite{lurie2014higher}*{Construction 5.2.1.1}, which we now briefly recall.

Let $I$ be a linearly ordered set. Then we denote by $I^\op$ the same set with the opposite ordering. Let $J$ be another linearly ordered set. Then we denote by $I\star J$ the set $I\sqcup J$ where the subsets $I$ and $J$ are equipped with their original ordering, and we have $i\le j$ for all $i\in I$ and $j\in J$. We can then define a functor
$$
\Delta\xrightarrow{Q(I)=I\star I^\op}\Delta.
$$
Let $\mathscr X$ be a quasi-category. Then the \emph{twisted arrow category} $\TwArr\mathscr X$ is the composite
$$
\Delta^\op\xrightarrow{Q^\op}\Delta^\op\xrightarrow{\mathscr X}\Set.
$$
By construction, its $n$-simplices are diagrams in $\mathscr X$ of the form
\[\begin{tikzcd}
	{X_0} & {X_1} & \dotsb & {X_n} \\
	\\
	{Y_0} & {Y_1} & \dotsb & {Y_n}
	\arrow[from=1-1, to=1-2]
	\arrow[from=1-2, to=1-3]
	\arrow[from=1-3, to=1-4]
	\arrow[from=1-4, to=3-4]
	\arrow[from=3-4, to=3-3]
	\arrow[from=3-3, to=3-2]
	\arrow[from=3-2, to=3-1]
	\arrow[from=1-1, to=3-1]
	\arrow[from=1-2, to=3-2]
\end{tikzcd}\]
Restricting along inclusion maps $I\to I\star I^\op\leftarrow I^\op$ induces a map $\TwArr\mathscr X\xrightarrow\lambda\mathscr X\times\mathscr X^\op$. This map is a right fibration \cite{lurie2014higher}*{Proposition 5.2.1.3}. It is classified by a presheaf $\mathscr X^\op\times\mathscr X\to\mathscr S$ which in turn induces the Yoneda embedding $\mathscr X\to\mathscr P(\mathscr X)$ \cite{lurie2014higher}*{Proposition 5.2.1.11}.

\end{defn}

\begin{cons} Since the inclusion $\{S\}\to\mathscr C$ is cofinal, the map $\{S\}\xrightarrow i\mathscr C^\sharp$ is a trivial Cartesian cofibration by Lemma \ref{lem:cofinal inclusion is trivial Cartesian cofibration}. Thus the base change $\mathscr X^\flat(S)\times\{S\}\xrightarrow{\mathscr X^\flat(S)\times i}\mathscr X^\flat\times\mathscr C$ is a trivial Cartesian cofibration too \cite{lurie2009higher}*{Corollary 3.1.4.3}. Therefore, the lifting problem
\[\begin{tikzcd}
	{\mathscr X^\flat(S)\times\{S\}} && {\mathscr X^\natural} \\
	\\
	{\mathscr X^\flat(S)\times\mathscr C^\sharp} && {\mathscr C^\sharp}
	\arrow["p"', from=1-3, to=3-3]
	\arrow[hook, from=1-1, to=1-3]
	\arrow["{\mathscr X^\flat(S)\times i}"', hook, from=1-1, to=3-1]
	\arrow["\pi", from=3-1, to=3-3]
	\arrow["C", dashed, from=3-1, to=1-3]
\end{tikzcd}\]
has a solution $\mathscr X^\flat(S)\times\mathscr C^\sharp\xrightarrow{C}\mathscr X^\natural$. That is, taking Cartesian resolutions can be done in a natural way.

Let $\mathscr X^\flat(S)\xrightarrow j\mathscr X$ denote the inclusion. Now we can take the strict fibre product
\begin{center}
	
	\begin{tikzpicture}[xscale=4,yscale=1.5]
	\node (C') at (0,1) {$\SMap$};
	\node (D') at (1,1) {$\TwArr(\mathscr X)$};
	\node (C) at (0,0) {$\mathscr X^\flat(S)\times\mathscr C\times\mathscr X^\flat(S)^\op$};
	\node (D) at (1,0) {$\mathscr X\times\mathscr X^\op.$};
	\path[->,font=\scriptsize,>=angle 90]
	(C') edge (D')
	(C') edge (C)
	(D') edge node [right] {$\lambda$} (D)
	(C) edge node [above] {$C\times j^\op$} (D);
	\end{tikzpicture}
	
\end{center}
Since $\lambda$ is a right fibration, its pullback $\lambda|C\times j$ is a right fibration too. Note that the fibre of $\SMap$ along the inclusion $\{x\}\times\mathscr C\times\{y\}\to\mathscr X^\flat(S)\times\mathscr C\times\mathscr X^\flat(S)^\op$ is equivalent to $\SMap(x,y)$ indeed.

\end{cons}

\begin{lem}\label{lem:cofinal inclusion is trivial Cartesian cofibration} Let $K$ be a simplicial set. Let $L'\xrightarrow iL$ be an inclusion of simplicial sets over $K$. Suppose that it is cofinal. Then the inclusion of marked simplicial sets $L^\sharp\xrightarrow{i^\sharp}(L')^\sharp$ is a trivial Cartesian cofibration over $K$.

\end{lem}

\begin{proof} It suffices to show that for all Cartesian fibrations $Y\to K$, the precomposition map
$$
\Map_K(L,Y^\simeq)=\Map^\sharp_K(L^\sharp,Y^\natural)
\xrightarrow{\circ i}\Map^\sharp_K((L')^\sharp,Y^\natural)=\Map_K(L',Y^\simeq)
$$
is a homotopy equivalence \cite{lurie2009higher}*{Proposition 3.1.3.3}. But this follows from that the map $i$ is cofinal \cite{lurie2009higher}*{Definition 4.1.1.1}.

\end{proof}

\begin{notn} In the case of main interest $\mathscr C=\Sch_X$ and $\mathscr X=\mathscr D^-$ we let $\SHom_X(E,F)=\SMap_{\mathscr D^-}(E,F)$.

\end{notn}

\begin{prop}\label{prop:Map is a stack} Let $\mathscr C$ be a quasi-category with a final object $S$, and $\mathscr X\xrightarrow p\mathscr C$ a Cartesian fibration. Let $U\xrightarrow gT$ be a morphism in $\mathscr C$. Suppose that $\mathscr X$ satisfies descent with respect to $g$. Let $x,y\in\mathscr X(S)$. Then the right fibration $\SMap(x,y)\to\mathscr C$ satisfies descent with respect to $g$ too.
	
\end{prop}

\begin{cor} \label{c:isStack}
	Let $X\xrightarrow fS$ be a morphism of schemes. Let $E,F\in\mathscr D(X)$. Then $\SHom(E,F)_X$ is a stack. Thus so is the pushforward $f_*\SHom(E,F)_X$. Note that the latter classifies families of maps: a section over an $S$-scheme $T$ classifies a map $E_{X_T}\to F_{X_T}$.
	
\end{cor}

\begin{proof}[Proof of Proposition \ref{prop:Map is a stack}] Since the construction of $\SMap(x,y)$ is natural, we can assume $T=S$. Let $U_\bullet=k$ denote the \v Cech nerve of $g$, and $U_\bullet^+=\Bar k$ the augmented \v Cech nerve. Let $x_\bullet$ resp.~$y_\bullet$ denote Cartesian resolutions of $x$ resp.~$y$. Let $U_\bullet\xrightarrow\varepsilon S$ denote the map in $\mathscr C^{\Delta^\op}$ induced by $U_\bullet^+$, and similarly for the map $U_\bullet^+\xrightarrow{\varepsilon^+}S$. The Cartesian resolution $y_\bullet$ induces morphisms $y_\bullet|k\xrightarrow{\varepsilon_y}y$ in $\mathscr X^{\Delta^\op}$ and $y_\bullet|\Bar k\xrightarrow{\varepsilon_y^+}y$ in $\mathscr X^{\Delta_+^\op}$. Since for each $n\in\Delta^\op$ the edge $y_n\xrightarrow{\varepsilon_y(n)}y$ is $p$-Cartesian, the edge $\varepsilon_y$ is $p^{\Delta^\op}$-Cartesian \cite{lurie2009higher}*{Proposition 3.1.2.1}. Therefore, the bottom horizontal map in the strict fibre product
	\begin{center}
		
		\begin{tikzpicture}[xscale=4,yscale=1.5]
		\node (C') at (0,1) {$\Gamma(U_\bullet,\mathscr X)^{/\varepsilon_y}$};
		\node (D') at (1,1) {$\Gamma(U_\bullet,\mathscr X)^{/ y}$};
		\node (C) at (0,0) {$(\mathscr X^{\Delta^\op})^{/\varepsilon_y}$};
		\node (D) at (1,0) {$(\mathscr X^{\Delta^\op})^{/y}\times_{(\mathscr C^{\Delta^\op})^{/S}}(\mathscr C^{\Delta^\op})^{/\varepsilon}$.};
		\path[->,font=\scriptsize,>=angle 90]
		(C') edge node [above] {$\phi$} (D')
		(C') edge (C)
		(D') edge node [right] {$(\mathrm{incl},*_\varepsilon)$} (D)
		(C) edge (D);
		\end{tikzpicture}
		
	\end{center}
	is a trivial fibration, and thus so is $\phi$. Since $\phi$ commutes with the restriction maps to $\Gamma(U_\bullet,\mathscr X)$, its restriction $\Gamma_{\Cart}(U_\bullet,\mathscr X)^{/\varepsilon_y}\xrightarrow\phi\Gamma_{\Cart}(U_\bullet,\mathscr X)^{/ y}$ is also a trivial fibration. The restriction map $\Gamma_{\Cart}(U_\bullet,\mathscr X)^{/\varepsilon_y}\to\Gamma_{\Cart}(U_\bullet,\mathscr X)^{/y_\bullet|k}$ is also a trivial fibration, as we can postcompose. Therefore, since we do the same for the augmented simplicial diagrams, we get a strict commutative diagram
	\begin{center}
		
		\begin{tikzpicture}[xscale=4,yscale=1.5]
		\node (B') at (-1,1) {$\Gamma_{\Cart}(U_\bullet^+,\mathscr X)^{/ y_\bullet|\Bar k}$};
		\node (C') at (0,1) {$\Gamma_{\Cart}(U_\bullet^+,\mathscr X)^{/\varepsilon_y^+}$};
		\node (D') at (1,1) {$\Gamma_{\Cart}(U_\bullet^+,\mathscr X)^{/ y}$};
		\node (B) at (-1,0) {$\Gamma_{\Cart}(U_\bullet,\mathscr X)^{/ y_\bullet|k}$};
		\node (C) at (0,0) {$\Gamma_{\Cart}(U_\bullet,\mathscr X)^{/\varepsilon_y}$};
		\node (D) at (1,0) {$\Gamma_{\Cart}(U_\bullet,\mathscr X)^{/ y}$};
		\path[->,font=\scriptsize,>=angle 90]
		(C') edge (B')
		(C') edge (D')
		(B') edge (B)
		(C') edge (C)
		(D') edge node [right] {$\psi$} (D)
		(C) edge (B)
		(C) edge (D);
		\end{tikzpicture}
		
	\end{center}
	where all the horizontal arrows are trivial fibrations. Since $\mathscr X$ satisfies descent with respect to $g$, the left vertical arrow is a trivial fibration too. Therefore, by the 2-out-of-3 property, the other two vertical arrows are categorical equivalences. Since the vertical arrows are right fibrations by Lemma \ref{lem:diagram restriction inner fibration}, they are Cartesian fibrations \cite{lurie2009higher}*{Proposition 2.4.2.4}. Therefore, they are trivial fibrations \cite{lurie2009higher}*{Corollary 2.4.4.6}.
	
	By construction, we have a strict fibre product diagram
	\begin{center}
		
		\begin{tikzpicture}[xscale=4,yscale=1.5]
		\node (C') at (0,1) {$\Gamma(U_\bullet,\SMap(x,y))$};
		\node (D') at (1,1) {$\Gamma(U_\bullet,\mathscr X)^{/y}$};
		\node (C) at (0,0) {$\Delta^0$};
		\node (D) at (1,0) {$\Gamma(U_\bullet,\mathscr X)$.};
		\path[->,font=\scriptsize,>=angle 90]
		(C') edge (D')
		(C') edge (C)
		(D') edge (D)
		(C) edge node [above] {$x_\bullet|k$} (D);
		\end{tikzpicture}
		
	\end{center}
	
	Therefore, we get a diagram with the square a strict fibre product
	\begin{center}
		
		\begin{tikzpicture}[xscale=4,yscale=1.5]
		\node (C'') at (0,2) {$\Gamma(U_\bullet^+,\SMap(x,y))$};
		\node (C') at (0,1) {$Z$};
		\node (D') at (1,1) {$\Gamma_{\Cart}(U_\bullet^+,\mathscr X)^{/y}$};
		\node (C) at (0,0) {$\Gamma(U_\bullet,\SMap(x,y))$};
		\node (D) at (1,0) {$\Gamma_{\Cart}(U_\bullet,\mathscr X)^{/y}$.};
		\path[->,font=\scriptsize,>=angle 90]
		(C'') edge node [right] {$\iota$} (C')
		(C') edge (D')
		(C') edge node [right] {$\psi$} (C)
		(D') edge node [right] {$\psi$} (D)
		(C) edge node [above] {$x_\bullet|k$} (D);
		\end{tikzpicture}
		
	\end{center}
	
	Since $\psi$ is a trivial fibration, $Z$ is a Kan complex. Therefore, $\iota$ is the inclusion map of a full Kan subcomplex of a Kan complex, the image of which intersects every connected component of $Z$. Therefore, it is a weak equivalence. This shows that $\psi\iota$ is a weak equivalence. But $\psi\iota$ is also a Kan fibration by Lemma \ref{lem:diagram restriction inner fibration}. Therefore, it is a trivial fibration as required.
	
\end{proof}

\begin{lem}\label{lem:diagram restriction inner fibration} Let $L'\subseteq L$ be simplicial sets, and $\mathscr X\xrightarrow p L$ an inner fibration. Then the restriction maps $\Gamma(L,\mathscr X)\to\Gamma(L',\mathscr X)$ and $\Gamma_{\Cart}(L,\mathscr X)\to\Gamma_{\Cart}(L',\mathscr X)$ are inner fibrations. If $p$ is a right fibration, then the restriction map is a Kan fibration.
	
\end{lem}

\begin{proof} Let $n\ge2$ and $0<i<n$. Then a lifting problem
	\begin{center}
		
		\begin{tikzpicture}[xscale=3,yscale=1.5]
		\node (C') at (0,1) {$\Lambda^n_d$};
		\node (D') at (1,1) {$\Gamma(L,\mathscr X)$};
		\node (C) at (0,0) {$\Delta^n$};
		\node (D) at (1,0) {$\Gamma(L',\mathscr X)$.};
		\path[->,font=\scriptsize,>=angle 90]
		(C') edge (D')
		(C') edge (C)
		(D') edge (D)
		(C) edge (D);
		\end{tikzpicture}
		
	\end{center}
	is the same as a lifting problem
	\begin{center}
		
		\begin{tikzpicture}[xscale=3,yscale=1.5]
		\node (C') at (0,1) {$(\Lambda^n_d\times L)\times_{\Lambda^n_d\times L'}(\Delta^n\times L')$};
		\node (D') at (1,1) {$\mathscr X$};
		\node (C) at (0,0) {$\Delta^n\times L$};
		\node (D) at (1,0) {$L$.};
		\path[->,font=\scriptsize,>=angle 90]
		(C') edge (D')
		(C') edge (C)
		(D') edge (D)
		(C) edge (D);
		\end{tikzpicture}
		
	\end{center}
	Thus, it is enough to show that the inclusion map $(\Lambda^n_d\times L)\times_{\Lambda^n_d\times L'}(\Delta^n\times L')\subseteq L\times\Delta^n$ is inner anodyne. This is implied by that $\Lambda^n_d\subset\Delta^n$ is inner anodyne \cite{lurie2009higher}*{Corollary 2.3.2.4}. As the subcategories on Cartesian sections are full subcategories, the same argument goes through for them. In case $p$ is a right fibration, we can use a similar argument for $n\ge0$ and $0\le d\le n$.
	
\end{proof}

We will now show that our two potentially different notions of automorphism group are indeed the same.
The proof will make use of the contravariant model structure on $\Set_\Delta/S$,
see \cite{lurie2009higher}*{2.1.4.12 and 2.2.1.2 }.

\begin{prop}\label{prop:Aut in Map} Let $X$ be a quasi-category. Let $\mathscr X\xrightarrow pX$ be a Cartesian fibration. Recall from Definition \ref{defn:interior} that $\mathscr X^\simeq\subseteq\mathscr X$ is the largest simplicial subset with all arrows $p$-Cartesian and the restriction $\mathscr X^\simeq\xrightarrow{p|\mathscr X^\simeq}X$ is a right fibration. Let $X\xrightarrow x\mathscr X^\simeq$ be a pointed object in the $\infty$-topos of right fibrations on $X$. Let $\Aut(x)'\subseteq\Map(x,x)$ denote the full substack on equivalences. Then we have a natural equivalence $\Aut(x)\simeq\Aut(x)'$.
	
\end{prop}

\begin{proof} Consider the commutative diagram of Cartesian fibrations over $X$, where, the back square is the defining square for the mapping stack, the skew arrows are inclusion maps, and thus every square is strict Cartesian:
	\begin{center}
		
		\begin{tikzpicture}[xscale=3,yscale=1.5]
		\node (000) at (0,2) {$\Aut(x)'$};
		\node (100) at (2,2) {$(\mathscr X^\simeq)^{/x}$};
		\node (010) at (0,0) {$X$};
		\node (110) at (2,0) {$\mathscr X^\simeq.$};
		\node (001) at (1,3) {${\SMap}(x,x)$};
		\node (101) at (3,3) {$\mathscr X^{/x}$};
		\node (011) at (1,1) {$X$};
		\node (111) at (3,1) {$\mathscr X$};
		\path[->,font=\scriptsize,>=angle 90]
		(001) edge (101)
		(001) edge (011)
		(101) edge (111)
		(011) edge (111)
		(000) edge (001)
		(100) edge (101)
		(010) edge (011)
		(110) edge (111)
		(000) edge [-,line width=6pt,draw=white] (100)
		(000) edge (100)
		(000) edge [-,line width=6pt,draw=white] (010)
		(000) edge (010)
		(010) edge [-,line width=6pt,draw=white] (110)
		(010) edge (110)
		(100) edge [-,line width=6pt,draw=white] (110)
		(100) edge (110);
		\end{tikzpicture}
		
	\end{center}
	In particular, the front square is strict Cartesian. The interior $\mathscr X^\simeq$ is a right fibration over $X$. The right front vertical arrow $(\mathscr X^\simeq)^{/x}\to\mathscr X^\simeq$ is a contravariant fibrant resolution of $X\xrightarrow x\mathscr X^\simeq$. Therefore, the front square is homotopy Cartesian. In particular, it gives a natural equivalence $\Aut(x)'\simeq\Aut(x)$.
	
\end{proof}

\begin{notn} 
	\label{n:rshom}
	We let $\SHom_{X/S}(E,F):=f_*\SHom_{X/S}(E,F)$ denote the pushforward.

\end{notn}

\begin{rem} 
	Recall that an $\infty$-group is a certain kind of simplicial object in an
	quasi-category, see \cite[pg. 718]{lurie2009higher}.
	Let's describe how this works in our case of interest $\mathscr X=\Perf_X$. Let $E$ be a perfect complex on $X$. We have, over an $X$-scheme $U$, we have $\SMap(E,E)(U)\simeq\DK\tau_{\le0}\REnd(E_U)$, and $\Aut_{\Perf}(E)(U)\subseteq\SMap(E,E)(U)$ is the full $\infty$-subcategory on quasi-isomorphisms. We can assume $X=U$. We will now show two ways to give $G=\Aut_{\Perf(X)}(E)$ an $\infty$-group structure.
	
	a) In \cite{dhillon2018stack}*{\S7}, we have shown that in general the loop group $\Omega(c)$ of a pointed object $*\xrightarrow cC$ in an $\infty$-topos $\mathscr C$ can be given as the vertical bisimplicial set
	$$
	\Delta^\op\times\Delta^\op\xrightarrow{G_m=G_{m,\bullet}}\Set
	$$
	where $G_{m,n}$ is the set of maps $\Delta^m\times\Delta^n\xrightarrow\sigma\mathscr C$ such that $\sigma|(\sk_0\Delta^m)\times\Delta^n$ factors through $c$, and the maps are given by precomposition. This means for example that
	
	i) $G_1$ classifies quasi-isomorphisms $E\to E$. That is, a vertex in $G_1$ is a quasi-isomorphism $E\to E$, an edge $\phi\xrightarrow H\psi$ in $G_1$ is a diagram
	\begin{center}
		
		\begin{tikzpicture}[xscale=3,yscale=1.5]
		\node (C') at (0,1) {$E$};
		\node (D') at (1,1) {$E$};
		\node (C) at (0,0) {$E$};
		\node (D) at (1,0) {$E$};
		\node at (0.5,0.5) {$\Searrow\,H$};
		\path[->,font=\scriptsize,>=angle 90]
		(C') edge node [above] {$\phi$} (D')
		(C') edge node [left] {$\id$} (C)
		(D') edge node [right] {$\id$} (D)
		(C) edge node [below] {$\psi$} (D);
		\end{tikzpicture}
		
	\end{center}
	in $\mathscr D(X)$, that is a homotopy $E\xrightarrow HE[-1]$ such that $d(H)=\psi-\phi$, etc.
	
	ii) $G_2$ classifies composition up to homotopy diagrams. That is, a vertex in $G_2$ is a diagram
	\begin{center}
		
		\begin{tikzpicture}[xscale=2,yscale=2]
		\node (A) at (-1,0) {$E$};
		\node (B) at (-0.33,0.5) {$E$};
		\node (C) at (1,0) {$E$};
		\node at (-0.25,0.15) {$\Downarrow\,\alpha$};
		\path[->,font=\scriptsize,>=angle 90]
		(A) edge node [left] {$\phi_{01}$} (B)
		(B) edge node [right] {$\phi_{12}$} (C)
		(A) edge node [below] {$\phi_{02}$} (C);
		\end{tikzpicture}
		
	\end{center}
	in $\mathscr D(X)$, that is the $\phi_{ij}$ are quasi-isomorphisms, and $E\xrightarrow\alpha E[-1]$ is a homotopy such that $d(\alpha)=\phi_{02}-\phi_{12}\circ\phi_{01}$, an edge $(\phi,\alpha)\xrightarrow H(\psi,\beta)$ is a diagram
	\begin{center}
		
		\begin{tikzpicture}[xscale=2,yscale=2]
		\node (A') at (-1,1) {$E$};
		\node (B') at (-0.33,1.5) {$E$};
		\node (C') at (1,1) {$E$};
		\node at (0.15,1.15) {$\Downarrow\,\alpha$};
		\node (A) at (-1,0) {$E$};
		\node (B) at (-0.33,0.5) {$E$};
		\node (C) at (1,0) {$E$};
		\node at (0.15,0.15) {$\Downarrow\,\beta$};
		\path[->,font=\scriptsize,>=angle 90]
		(A) edge node [left] {$\psi_{01}$} (B)
		(B) edge node [right] {$\psi_{12}$} (C)
		(A) edge node [below] {$\psi_{02}$} (C)
		(A') edge node [left] {$\id$} (A)
		(B') edge node [pos=0.7, below left] {$\id$} (B)
		(C') edge node [right] {$\id$} (C)
		(A') edge node [left] {$\phi_{01}$} (B')
		(B') edge node [right] {$\phi_{12}$} (C')
		(A') edge [-,line width=6pt,draw=white] (C')
		(A') edge node [below] {$\phi_{02}$} (C');
		\end{tikzpicture}
		
	\end{center}
	in $\mathscr D(X)$, that is we have 1-homotopies
	$$\phi_{12}\circ\phi_{01}\xrightarrow\alpha\phi_{02},\,\psi_{12}\circ\psi_{01}\xrightarrow\beta\psi_{02},$$
	$$\phi_{01}\xrightarrow{H_{01}}\psi_{01},\,\phi_{02}\xrightarrow{H_{02}}\psi_{02},\,\phi_{12}\xrightarrow{H_{12}}\psi_{12},$$
	and a 2-homotopy $E\xrightarrow\Theta E[-2]$ such that 
	$$
	d(\Theta)=-H_{02}+H_{12}\circ\phi_{01}+\psi_{12}\circ H_{01}+\beta-\alpha,
	$$
	etc.
	
	iii) $G_3$ classifies associativity up to homotopy diagrams, and so on.
	
	b) The complex $\REnd(E)$ has a strict monoid structure, giving rise to a simplicial object in the category of complexes. Let $\RAut(E)$ denote the subcomplex
	$$
	\RAut(E)_n=\begin{cases}
	0 & n > 0\\
	\{\text{quasi-automorphisms }E\to E\} & n = 0 \\
	\REnd(E)_n & n < 0.
	\end{cases}
	$$
	Applying the functor $\DK$ to $\RAut(E)$, we get a simplicial object $\Aut''(E)$ of Kan complexes. We claim that this is equivalent to $\Aut_{\Perf}(E)$. Since $\RAut(E)$ with composition is a monoid and $\DK$ is an equivalence, $\Aut''(E)$ is a group object. Therefore, it remains to show that the square
	\begin{center}
		
		\begin{tikzpicture}[xscale=3,yscale=1.5]
		\node (C') at (0,1) {$\DK\RAut(E)$};
		\node (D') at (1,1) {$*$};
		\node (C) at (0,0) {$*$};
		\node (D) at (1,0) {$\Perf(X)$};
		\path[->,font=\scriptsize,>=angle 90]
		(C') edge (D')
		(C') edge (C)
		(D') edge node [right] {$\mathbf c_E$} (D)
		(C) edge node [below] {$\mathbf c_E$} (D);
		\end{tikzpicture}
		
	\end{center}
	is homotopy Cartesian. 
	We're done by Proposition \ref{prop:Aut in Map}.
	
\end{rem}


\subsection{$\Aut E$ is algebraic}

Let $k$ be a field, $X\xrightarrow fS$ a proper morphism of schemes, and $E,F$ perfect complexes on $X$. Our goal in this
section is to show that $\SHom_{X/S}(E,F)$ is algebraic. 
This will imply that $\Aut_{X/S}(E)$ and $\B\Aut_{X/S}(E)$ are also algebraic. In \cite{toen2007moduli}*{Corollary 3.29}, To\"en and Vaqui\'e show that in case $f$ is smooth and proper, the stack $\Perf_{X/S}^{[a,b]}$ of families of perfect complexes of a given Tor amplitude is algebraic. Note that we do not require $f$ to be smooth.

Let's start by recalling some definitions.
This discussion is borrowed from  \cite[Ch. 2, \S4]{gr}.

Given a non-negative integer $n$ we will inductively define 
\begin{enumerate}
	\item what it means for $\msX\in\St(S)$	to be an $n$-algebraic stack
	\item what it means for a morphism $f:\msX\rightarrow \msY$ in $\St(S)$ to be $n$-representable
	\item what it means for an $n$-representable morphism to be smooth, surjective, \'etale and flat.
\end{enumerate}

The collection of algebraic stacks and representable morphisms is obtained by taking the union over all
$n$.

A $0$-algebraic stack is a disjoint union of affine schemes and a $0$-representable morphism 
is a morphism $f:\msX\rightarrow \msY$ such that for every
$$
T\rightarrow \msY,
$$
where $T$ is a disjoint union of affine schemes, the fibered product $T\times_\msY \msX$ is 
a disjoint union of affine schemes. The notions of smooth, surjective, and flat have their usual meanings.

An $n$-algebraic stack is a stack $\msX\in\St(S)$ so that 
\begin{enumerate}
	\item the diagonal $\msX\rightarrow \msX\times \msX$ is $(n-1)$-representable,
	\item and there is an $(n-1)$-algebraic stack $P$ and a smooth $(n-1)$-representable surjective morphism
	$p:P\rightarrow \msX$. (See below for the meaning of smooth and surjective, defined inductively)
\end{enumerate}
One can check, using the first condition, that the morphism $p$ above is always $(n-1)$-representable. The morphism $p$
is called a presentation for $\msX$. It allows us to transport notions from algebraic geometry to the study of
$\msX$. For example we call $\msX$ smooth if $P$ is smooth, defined inductively. 

We say that a morphism $f:\msX\rightarrow \msY$ is $n$-representable if for every 
$$
T\rightarrow \msY,
$$
where $T$ is a disjoint union of affine schemes, the fibered product $T\times_\msY \msX$ is 
an $n$-algebraic stack.
We say that the $n$-representable morphism is smooth (resp. surjective or flat) if for every affine scheme $T$ and morphism $T\rightarrow \msY$ the composite morphism
$$
P\stackrel{p}{\rightarrow} T\times_\msY \msX \rightarrow T,
$$
where $p$ is a presentation, is smooth (resp. surjective or flat). This does not depend on the choice of 
presentation, see \cite[3.6.7]{gr}.

\begin{prop}\label{p:algebraic}
	\begin{enumerate}
	\item
    Let  $f:\msX\rightarrow \msY$ be an $n$-representable morphism in $\St(S)$.	Then the diagonal
    map $ \msX \rightarrow \msX\times_\msY \msX$ is $(n-1)$-representable.
    \item
	Let $f:\msX\rightarrow \msY$ be a $n$-representable morphism of stacks. Suppose that
	$\msY$ is $n$-algebraic. Then $\msX$ is $n$-algebraic.
	\end{enumerate}
\end{prop}

\begin{proof}
	This is \cite[Lemma 4.2.2, Proposition 4.2.4]{gr}. We will sketch the idea here.
	
	We start with the first part.
	Consider a morphism from an affine scheme $T\rightarrow \msX\times_\msY\msX$ given by a pair of morphisms
	$\alpha$ and $\beta$. Composing $\alpha$ and $\beta$ with $f$, we get a pair of equivalent morphisms $T\rightarrow \msY$ making $T$ into a $\msY$-scheme. Write $\msX_T$ for $\msX\times_\msY T$. Since $f$ is $n$-representable, the pullback $\mathscr X_T$ is $n$-algebraic over $T$. Therefore, the diagonal
	$$
	\mathscr X_T\to\mathscr X_T\times_T\mathscr X_T
	$$
	is $(n-1)$-representable. Therefore its pullback along $T\to\mathscr X_T\times_T\mathscr X_T$ is $(n-1)$-algebraic. But this is equivalent to the pullback of the diagonal $\mathscr X\to\mathscr X\times_{\mathscr Y}\mathscr X$ along $T\to\mathscr X\times_{\mathscr Y}\mathscr X$, thus concluding the proof.
	
	The second part is proved in detail in \cite[4.2.4]{gr}, so we sketch the proof. One factors the
	diagonal as
	$$
	\msX\rightarrow \msX\times_\msY\msX \rightarrow \msX\times\msX. 
	$$
	The first map is $(n-1)$-representable. The morphism $\msY\rightarrow \msY\times\msY$ is also $(n-1)$-representable. This implies that  $\msX\times_\msY\msX \rightarrow \msX\times\msX$ is $(n-1)$-representable
	as it is a base change of the diagonal and one can check that a base change of a $(n-1)$-representable
	morphism is $(n-1)$-representable by pasting cartesian squares.
	
	Finally we need a presentation for $\msX$. Let $P\rightarrow \msY$ be a presentation. Then
	$P\times_\msY \msX$ is $n$-algebraic so has a presentation $P'$ which is also a presentation for $\msX$.
\end{proof}

\begin{defn} Let $E$ be a perfect complex on $X$. We will say that \emph{its Tor amplitude is contained in $[a,b]$}, if for all $\mathscr O_X$-modules $M$, we have
$$
\mathscr H^i(E\otimes^LM)=0\text{ for }i\notin[a,b].
$$
We will also say that \emph{$E$ is of Tor amplitude $[a,b]$} if $[a,b]$ is the minimal interval so that the Tor amplitude of $E$ is contained in $[a,b]$. We will say that \emph{$E$ has Tor length $b-a$}, if it has Tor amplitude $[a,b]$.

\end{defn}

\begin{prop}\label{prop:Tor amplitude} The following assertions hold.

\begin{enumerate}

\item There exist integers $a\le b$ such that $E$ is of Tor amplitude $[a,b]$.

\item Let $Y\xrightarrow gX$ be a morphism of $S$-schemes. Let $E$ be a perfect complex of Tor amplitude $[a,b]$. Then $Lg^*E$ is of Tor amplitude $[a,b]$.

\item Let $Y\to X$ be a morphism of $S$-schemes. Let $\Perf^{\ge0}(Y)\subset\Perf(X)$ denote the full $\infty$-subcategory of perfect complexes with Tor amplitude contained in $[0,b]$ for some $b\ge0$. Similarly, let $\Perf^{\le0}(Y)\subset\Perf(X)$ denote the full $\infty$-subcategory of perfect complexes with Tor amplitude contained in $[a,0]$ for some $a\le0$. Then $(\Perf^{\le0}(Y),\Perf^{\ge0}(Y))$ is a $t$-structure on the stable quasi-category $\Perf(Y)$.

\end{enumerate}

\end{prop}

\begin{proof} Since we're using cohomology sheaves in the definition of Tor amplitude, these follow from \cite{toen2007moduli}*{Proposition 2.22}.

\end{proof}

\begin{notn} In this section, the notations $\tau_{<b},\tau_{\le b},\tau_{>a}$ and $\tau_{\ge a}$ will refer to truncation with respect to the $t$-structure given by Tor amplitude.

\end{notn}

\begin{thm}\label{thm:RHom is algebraic} Let $X\xrightarrow fS$ be a proper morphism of quasi-compact and quasi-separated schemes, and $E,F$ perfect complexes on $X$. Then the stack of families of morphisms $\SHom_{X/S}(E,F)$ is algebraic.

\end{thm}

\begin{cor}\label{cor:Aut is algebraic} The automorphism group $\Aut_{X/S} E$ is algebraic.

\end{cor}

\begin{cor}\label{cor:BAut is algebraic} The delooping $\B\Aut_{X/S} E$ is algebraic.

\end{cor}

\begin{proof}[Proof of Corollary \ref{cor:Aut is algebraic}] By Proposition \ref{prop:Aut in Map}, $\Aut_{X/S}(E)\subseteq\SHom_{X/S}(E,E)$ is the full substack of quasi-isomorphisms. We claim that the inclusion map $\Aut_{X/S} E\xrightarrow i\SHom_{X/S}(E,E)$ is an open immersion. Let $T$ be an $S$-scheme, and $T\xrightarrow{t}\SREnd_{X/S}(E,F)$ a section. Then $T$ represents a morphism of complexes $E_{X_T}\xrightarrow{\phi}F_{X_T}$. By construction, the strict pullback $T'$ of $T$ along $i$ is the complement of the support of the cone of $\phi$, that is the intersection of the complements of the supports of the cohomology sheaves of the cone of $\phi$, which is closed. Since fibrewise $i$ is the inclusion of connected components, $T'$ is the homotopy pullback of $T$ along $i$ too.

\end{proof}

\begin{proof}[Proof of Theorem \ref{thm:RHom is algebraic}] We can assume that $E,F\ne0$. Let us use induction on the sum $\ell+m$ of the Tor lengths of $E$ and $F$. By Lemma \ref{lem:RHom and shift}, for the starting case $\ell+m=2$ we can assume that $E$ and $F$ are locally free sheaves of finite rank concentrated in degree 0. But then $\SHom_{X/S}(E,F)$ is an algebraic space.

For the inductive case, we can suppose that $E$ is of Tor amplitude $[a,b]$ with $a<b$, as the case $\ell>1$ can be dealt with in a similar manner. Let $T=\SHom_{X/S}(E,\tau_{>a}F)$. We have
$$
\SHom_{X_T/T}(E_T,F_T)\simeq\SHom_{X/S}(E,F)\times_ST$$
{ and }
$$\SHom_{X_T/T}(E_T,\tau_{>a}F_T)\simeq\SHom_{X/S}(E,\tau_{>a}F)\times_ST.
$$
Therefore, we obtain a square of $T$-stacks
\begin{center}

\begin{tikzpicture}[xscale=5,yscale=1.5]
\node (C') at (0,1) {$\SHom_{X/S}(E,F)$};
\node (E') at (1,1) {$\SHom_{X/S}(E,F)\times_ST$};
\node at (1.5,1) {$\simeq$};
\node (D') at (2,1) {$\SHom_{X_T/T}(E_T,F_T)$};
\node (C) at (0,0) {$T$};
\node (E) at (1,0) {$\SHom_{X/S}(E,\tau_{>a}F)\times_ST$};
\node at (1.5,0) {$\simeq$};
\node (D) at (2,0) {$\SHom_{X_T/T}(E_T,\tau_{>a}F_T),$};
\path[->,font=\scriptsize,>=angle 90]
(C') edge node [above] {$\Gamma_{\tau_{>a}}$} (E')
(C') edge node [right] {$\tau_{>a}$} (C)
(D') edge node [right] {$\tau_{>a}$} (D)
(C) edge node [above] {$\Delta$} (E);
\end{tikzpicture}

\end{center}
where the top horizontal arrow is the graph of the truncation map
$$
\tau_{>a} : \SHom_{X/S}(E,F)\rightarrow T=\SHom_{X/S}(E,\tau_{>a}F).
$$
This Cartesian square is in fact homotopy 
 Cartesian. To see this observe that we have a diagram 
\begin{center}
	
	\begin{tikzpicture}[xscale=5,yscale=1.5]
	\matrix (m) [matrix of math nodes,row sep=3em,column sep=4em,minimum width=2em]
	{
		\SHom_{X/S}(E,F) & \SHom(E,F)\times_S \SHom(E,\tau_{>a} F) & \SHom(E,F) \\
	 	\SHom_{X/S}(E,\tau_{>a}F) & \SHom(E,\tau_{>a} F)\times_S \SHom(E,\tau_{>a} F) & \SHom(E,\tau_{>a}F) \\};
 	\path[-stealth]
 	(m-1-1) edge (m-1-2) edge (m-2-1) 
 	(m-2-1) edge (m-2-2)
 	(m-1-2) edge (m-1-3) edge (m-2-2)
 	(m-2-2) edge (m-2-3)
 	(m-1-3) edge (m-2-3);
	\end{tikzpicture}
\end{center} 
in which the right horizontal arrows are projections. As the $\SHom_{X/S}(E,F)$ are stacks they are fibrant and hence the bottom right horizontal arrow is a
fibration. It follows that the right square is homotopy Cartesian. As the outer square is clearly so, the claim follows from the pasting lemma.
 
By via an argument similar to that of Lemma \ref{lem:loop group of RHom}, we also have the homotopy Cartesian diagram
\begin{center}

\begin{tikzpicture}[xscale=5,yscale=1.5]
\node (C') at (0,1) {$\SHom_{X_T/T}(E_T,\tau_{\le a}F_T)$};
\node (D') at (1,1) {$\SHom_{X_T/T}(E_T,F_T)$};
\node (C) at (0,0) {$T$};
\node (D) at (1,0) {$\SHom_{X_T/T}(E_T,\tau_{>a}F_T).$};
\path[->,font=\scriptsize,>=angle 90]
(C') edge node [above] {$\tau_{\le a}\circ$} (D')
(C') edge (C)
(D') edge node [right] {$\tau_{>a}\circ$} (D)
(C) edge node [above] {$\Delta$} (D);
\end{tikzpicture}

\end{center}
Therefore, the map of $S$-stacks $\SHom_{X/S}(E,F)\to\SHom_{X/S}(E,\tau_{>a}F)$ is equivalent to the map of $S$-stacks $\SHom_{X_T/T}(E_T,\tau_{\le a}F_T)\to T$, and thus representable. But $T\to S$ is algebraic by the induction hypothesis, so,
using \ref{p:algebraic}, the composite $\SHom_{X/S}(E,F)\to S$ is algebraic too.
\end{proof}

\begin{lem}\label{lem:RHom and shift} Let $E\in\Perf(X)^{\ge0}$ and $F\in\Perf(X)^{<0}$ (recall that this means $E$ has Tor amplitude contained in $[0,\infty)$, and $F$ has Tor amplitude contained in $(-\infty,0)$. Then there exists an $\infty$-group structure on $\SHom_{X/S}(E,F[1])$ such that $\B(\SHom_{X/S}(E,F[1]))\simeq\SHom_{X/S}(E,F)$.

\end{lem}

\begin{proof} By Lemma \ref{lem:loop group of RHom}, the \v Cech nerve of the map $\ast\xrightarrow0\SHom_{X/S}(E,F)$ gives the group structure on $\SHom_{X/S}(E,F[1])$. Therefore, it is enough to show that $\ast\xrightarrow0\SHom_{X/S}(E,F)$ is an effective epimorphism \cite{lurie2009higher}*{below Corollary 6.2.3.5}. It is enough to show that for all $T\in\Sch_S$, we have $\pi_0(\SHom_{X/S}(E,F)(T))=\Hom_{D(X_T)}(E_T,F_T)=0$. Note that $D(X_T)$ is the derived 1-category. By definition of Tor amplitude, we can assume $S=T$. Let $E\xrightarrow\alpha F$ be a morphism. Let us prove $\Hom_{D(X)}(E,F)=0$. By assumption, we have $\tau_{<0}E=0$ and $\tau_{\ge0}F=0$. Therefore, we get a morphism of distinguished triangles
\begin{center}

\begin{tikzpicture}[xscale=2,yscale=1.5]
\node (C') at (0,1) {$\tau_{<0}E$};
\node (D') at (1,1) {$E$};
\node (E') at (2,1) {$\tau_{\ge0}E$};
\node (C) at (0,0) {$\tau_{<0}F$};
\node (D) at (1,0) {$F$};
\node (E) at (2,0) {$\tau_{\ge0}F.$};
\path[->,font=\scriptsize,>=angle 90]
(C') edge node [above] {$\tau_{<0}$} (D')
(C') edge node [right] {$0$} (C)
(D') edge node [above] {$\tau_{\ge0}$} (E')
(D') edge node [right] {$\alpha$} (D)
(E') edge node [right] {$0$} (E)
(C) edge node [above] {$\tau_{<0}$} (D)
(D) edge node [above] {$\tau_{\ge0}$} (E);
\end{tikzpicture}

\end{center}
This shows that $\alpha=0$ as needed.

\end{proof}

\begin{lem}\label{lem:loop group of RHom} Let $E,F\in\mathscr D(X)$. Then the diagram of $\infty$-stacks
\begin{center}

\begin{tikzpicture}[xscale=3,yscale=1.5]
\node (C') at (0,1) {$\SHom_{X/S}(E,F)$};
\node (D') at (1,1) {$\ast$};
\node (C) at (0,0) {$\ast$};
\node (D) at (1,0) {$\SHom_{X/S}(E[1],F)$};
\path[->,font=\scriptsize,>=angle 90]
(C') edge (D')
(C') edge (C)
(D') edge (D)
(C) edge (D);
\end{tikzpicture}

\end{center}
is a homotopy fibre product.

\end{lem}

\begin{proof} Since the vertices of the diagram are $\infty$-stacks, it is enough to show that the diagram is a homotopy fibre product of right fibrations. For that, it is enough to show that we have a homotopy fibre product of Kan complexes fibrewise. Since for $T\in\Sch_S$ we have a natural equivalence $\SHom_{X/S}(E,F)(T)\simeq\Map_{\mathscr D(X_T)^\op}(F_{X_T},E_{X_T})$, it is enough to show that the diagram of Kan complexes
\begin{center}

\begin{tikzpicture}[xscale=3,yscale=1.5]
\node (C') at (0,1) {$\Map_{\mathscr D(X_T)}(E_{X_T},F_{X_T})$};
\node (D') at (1,1) {$\ast$};
\node (C) at (0,0) {$\ast$};
\node (D) at (1,0) {$\Map_{\mathscr D(X_T)}(E_{X_T},F_{X_T}[1])$};
\path[->,font=\scriptsize,>=angle 90]
(C') edge (D')
(C') edge (C)
(D') edge (D)
(C) edge (D);
\end{tikzpicture}

\end{center}
is a homotopy fibre product. By the homotopical Dold--Kan correspondence, this in turn is equivalent to that the diagram of complexes of $\mathscr O_{X_T}$-modules
$$
\tau_{\le0}\RHom_{\mathscr D({X_T})}(E_{X_T},F_{X_T})\to0\to\tau_{\le0}\RHom_{\mathscr D({X_T})}(E_{X_T},F_{X_T}[1])
$$
is a fibration sequence. Since $\tau_{\le0}$ is a right adjoint and thus exact, it is enough to show that the diagram of complexes of $\mathscr O_{X_T}$-modules
$$
\RHom_{\mathscr D({X_T})}(E_{X_T},F_{X_T})\to0\to\RHom_{\mathscr D({X_T})}(E_{X_T},F_{X_T}[1])
$$
is a fibration sequence. This in turn follows from $\RHom_{\mathscr D({X_T})}(E_{X_T},F_{X_T}[1])=\RHom_{\mathscr D({X_T})}(E_{X_T},F_{X_T})[1]$.

\end{proof}

\subsection{The stack of forms $\B\Aut E$} Let $\mathscr X$ be an $\infty$-topos \cite{lurie2009higher}*{\S6.1}, $X\in \mathscr X$, and $\ast\xrightarrow xX$. In our case, $\mathscr X$ is the quasi-category of fppf stacks over a scheme $S$, modelled as the full $\infty$-subcategory of the quasi-category of right fibrations over $\mathscr C=\Sch_S$. Then $X$ is a stack, and $x$ is a global section. We would like to define the stack of forms of $x$.

To go about this in the setting of quasi-categories, we take the essential image of the classifying map $\ast\xrightarrow xX$. That is, one can first form the loop group $\Omega x$, which we will refer to as the automorphism group $\Aut x$. This will be a group object $\Delta^{\op}\to\mathscr X$ with underlying category $(\Aut x)_1=\ast\times_X\ast$. Since in an $\infty$-topos every groupoid object is effective, we can define the stack of forms of $x$ as the delooping $\B\Aut x$, that is the homotopy colimit of the automorphism group. By construction, the classifying map $x$ factors through the canonical map $\ast\to\B\Aut x$.

One can show that the factorizing map $\B\Aut x\to X$ is a monomorphism, that is this is the epi-mono factorization. In other words, to describe $\B\Aut X$, we only need to describe its objects. In the following, we show that the objects are precisely the forms of $x$, that is sections of $X$ which are locally equivalent to $x$.

\begin{prop} \label{prop:BAut classifies forms} Let $\mathscr D\in\mathscr X$ be an object in an $\infty$-topos, and $\ast\xrightarrow x\mathscr D$ a global section classifying $E\in\mathscr D(\ast)$. Let $T\xrightarrow y\mathscr D$ be another section classifying $F\in\mathscr D(T)$. Then $y$ maps into $\B\Aut E$ if and only if there exists an effective epimorphism $U\twoheadrightarrow T$ and an equivalence $E_U\simeq F_U$.

\end{prop}

\begin{proof} $\Rightarrow$: Suppose that $y$ maps into $\B\Aut E$. Then we can take the homotopy fibre product
\begin{center}

\begin{tikzpicture}[xscale=3,yscale=1.5]
\node (C') at (0,1) {$U$};
\node (D') at (1,1) {$\ast$};
\node (C) at (0,0) {$T$};
\node (D) at (1,0) {$\B\Aut E.$};
\path[->,font=\scriptsize,>=angle 90]
(C') edge (D')
(C') edge [->>] (C)
(D') edge [->>] node [right] {$x$} (D)
(C) edge node [above] {$y$} (D);
\end{tikzpicture}

\end{center}
The composite of the top and right arrows classifies $E|U$. The composite of the left and bottom arrows classifies $F|U$. Thus, the homotopies filling in the rectangle give an equivalence $E_U\simeq F_U$.

$\Leftarrow$: By definition, the essentially surjective map $U\twoheadrightarrow T$ is the homotopy colimit of its \v Cech nerve $\Delta^\op\xrightarrow{U_\bullet}\mathscr X$. Moreover, as we recalled it in Definition \ref{d:loop}, the delooping $\B\Aut E$ is the homotopy colimit of the simplicial object $\Aut(E)_\bullet$. Therefore, by the universal property of homotopy colimits, the map $y$ factors through the inclusion $\B\Aut E\subseteq\mathscr D$ if we can construct a morphism of augmented simplicial objects of the form:
\begin{center}

\begin{tikzpicture}[xscale=1.5,yscale=1.5]
\node (F') at (-2,1) {$\dots$};
\node (F) at (-2,0) {$\dots$};
\node (E') at (-1,1) {$U_1$};
\node (E) at (-1,0) {$\Aut E$};
\node (C') at (0,1) {$U$};
\node (D') at (1,1) {$T$};
\node (C) at (0,0) {$\ast$};
\node (D) at (1,0) {$\mathscr D.$};
\path[->,font=\scriptsize,>=angle 90]
(F') edge (E')
(E') edge (E)
(F) edge (E)
(E') edge (C')
(E) edge (C)
(C') edge (D')
(C') edge (C)
(D') edge node [right] {$y$} (D)
(C) edge node [below] {$x$} (D);
\end{tikzpicture}

\end{center}
The equivalence $E|U\simeq F|U$ gives a filler of the rightmost rectangle. Then we can complete the diagram by taking pullbacks. 

\end{proof}

In \cite{toen2009descente}*{Th\'eor\`eme 2.1} it is shown that if an $S$-stack $\mathscr F$ has an fppf-atlas, then it also has a smooth atlas. This gives us the following Corollary, which is the foundation of our Hilbert 90-type theorem.

\begin{cor} \label{c:smooth} \label{cor:classifying map and forms}Let $E$ be a perfect complex on $S$. Let $\mathbf P\in\{\text{fppf},\text{smooth}\}$. Suppose that $\Aut E$ is $\mathbf P$. Let $F$ be a complex on an $S$-scheme $T$. Then $F\in(\B\Aut E)(T)$ if and only if there exists a $\mathbf P$-covering $U\twoheadrightarrow T$ and a quasi-isomorphism $E_U\simeq F_U$.

\end{cor}

\begin{proof} It is enough to show necessity. Consider the following diagram with homotopy Cartesian squares.

\begin{center}

\begin{tikzpicture}[xscale=2,yscale=1.5]
\node (C') at (0,1) {$\Aut E$};
\node (D') at (1,1) {$\ast$};
\node (E') at (2,1) {$\mathscr U$};
\node (C) at (0,0) {$\ast$};
\node (D) at (1,0) {$\B\Aut E$};
\node (E) at (2,0) {$T.$};
\path[->,font=\scriptsize,>=angle 90]
(C') edge (D')
(C') edge (C)
(D') edge (D)
(E') edge (D')
(E') edge (E)
(C) edge [->>] node [above] {$x$} (D)
(E) edge node [above] {$y$} (D);
\end{tikzpicture}

\end{center}
Since by assumption the left vertical arrow is $\mathbf P$, and $x$ is an effective epimorphism, the middle vertical arrow is also $\mathbf P$. This in turn shows that the right vertical arrow is $\mathbf P$. Let $U\twoheadrightarrow\mathscr U$ be a smooth atlas \cite{toen2009descente}*{Th\'eor\`eme 2.1}. Then the composite $U\to T$ is a $\mathbf P$-cover, and we have $E_U\simeq F_U$.

\end{proof}


\section{Deformation theory of $\Aut E$}\label{s:deformation theory}

Let $C$ be a triangulated category. Given objects $A$ and $B$ we define
\[
\Ext^{1}(A,B) = \Hom(A,B[1]).
\]
A distinguished triangle 
\[
B\rightarrow T\rightarrow A\rightarrow B[1]
\]
determines an element of $\Ext^{1}(A,B)$ and conversely the axioms of a triangulated
category tell us that given $\sigma\in\Ext^1(A,B)$ there is a unique up to isomorphism triangle
determined by $\sigma$. We will say in this situation that $\sigma$ classifies the triangle.
Consider an extension in $C$, that is an exact triangle
\[
M\rightarrow E\xrightarrow{q} E_0\rightarrow M[1],
\]
and a morphism $ M\stackrel{f}{\rightarrow} N$. Then there exists a new triangle
 $ N\rightarrow f_*(E){\rightarrow}E_0\rightarrow N[1] $ and a morphism of triangles
 \begin{center}
 \begin{tikzpicture}
 \matrix (m) [matrix of math nodes,row sep=3em,column sep=4em,minimum width=2em] {
 	M  & E& E_0 & M[1]\\ 
 	N& f_*(E) & E_0 & N[1].\\};
 
 \path[->] (m-1-1)  edge(m-1-2);
 
 \path[->] (m-1-1)  edge(m-2-1);
 
 \path[->] (m-1-2)  edge(m-1-3);
 
 \path[->] (m-1-2)  edge(m-2-2);
 
 \path[->] (m-1-3)  edge(m-2-3);

 \path[->] (m-2-1)  edge(m-2-2);
 
 \path[->] (m-2-2)  edge(m-2-3);
 
 \path[->] (m-1-3)  edge(m-1-4);
 \path[->] (m-2-3)  edge(m-2-4);
 \path[->] (m-1-4)  edge(m-2-4);
 
 \end{tikzpicture}
 	
 \end{center}
 	
 	It is classified by a morphism in $\Hom(E_0,N[1])=\Ext^1(E_0,N)$
 	
\begin{lemma}
	In the above situation, the diagram
	\begin{center}
		\begin{tikzpicture}
		\matrix (m) [matrix of math nodes,row sep=3em,column sep=4em,minimum width=2em] {
			M & E  \\
			N & f_*(E) \\};
		
		\path[->] (m-1-1)  edge(m-1-2);
		
		\path[->] (m-1-1)  edge(m-2-1);
		
		\path[->] (m-1-2)  edge(m-2-2);
		
		\path[->] (m-2-1)  edge(m-2-2);
		
		\end{tikzpicture}
		
	\end{center}
is a push out.
\end{lemma}

\begin{proof}
This follows from the fact that 
\[
T\xrightarrow{1_T}T\rightarrow 0\rightarrow T[1]
\]
is a distinguished triangle.
\end{proof}

Given an exact triangle
\[
N\rightarrow F \rightarrow F_0\rightarrow N[1]\quad\mbox{and}\quad \lambda:E_0\rightarrow F_0
\]
we can consider the pullback triangle
\[
N\rightarrow\lambda^*F\rightarrow E_0\rightarrow N[1].
\]
It is classified by an element of $\Ext^1(E_0,N)$.

\begin{prop}
	Consider the diagram
	\begin{center}
	\begin{tikzpicture}
		\matrix (m) [matrix of math nodes,row sep=3em,column sep=4em,minimum width=2em] {
			M& E& E_0 & M[1]\\  
			N& F& F_0 & N[1]\\};
		\path[->] (m-1-1)  edge(m-1-2);
		\path[->] (m-1-2)  edge(m-1-3);
		\path[->] (m-1-3)  edge(m-1-4);
		
		\path[->] (m-2-1)  edge(m-2-2);
		\path[->] (m-2-2)  edge(m-2-3);
		\path[->] (m-2-3)  edge(m-2-4);

		\path[->] (m-1-1)  edge(m-2-1);
		\path[dashed, ->] (m-1-2)  edge(m-2-2);
		\path[->] (m-1-3)  edge(m-2-3);
		\path[dashed,->] (m-1-4)  edge(m-2-4);

		\end{tikzpicture}
		
	\end{center}
Then the dashed arrows exists if and only if $\lambda^*({F})\cong f_*({E})$ inducing a
diagram 
\begin{center}
	\begin{tikzpicture}
\matrix (m) [matrix of math nodes,row sep=3em,column sep=4em,minimum width=2em] {
	N& f_*E& E_0 & N[1]\\  
	N& \lambda^*F & E_0 & N[1]\\};
\path[->] (m-1-1)  edge(m-1-2);
\path[->] (m-1-2)  edge(m-1-3);
\path[->] (m-1-3)  edge(m-1-4);

\path[->] (m-2-1)  edge(m-2-2);
\path[->] (m-2-2)  edge(m-2-3);
\path[->] (m-2-3)  edge(m-2-4);

\path[transform canvas={xshift=0.2em}] (m-1-1)  edge(m-2-1);
\path (m-1-1)  edge(m-2-1);
\path[->] (m-1-2)  edge(m-2-2);

\path[transform canvas={xshift=0.2em}] (m-1-3)  edge(m-2-3);
\path (m-1-3)  edge(m-2-3);

\path[transform canvas={xshift=0.2em}] (m-1-4)  edge(m-2-4);
\path (m-1-4)  edge(m-2-4);
\end{tikzpicture}

\end{center}
\end{prop}

\begin{proof}
	If the dashed arrow exists then the left square is a pushout and the right a pullback.
	
	The converse follows by pasting triangles.
\end{proof}

\begin{cor} In the situation of the previous proposition, there is an obstruction
	\[
	[\lambda^*F]-[f_*(E)]\in \Ext^1(E_0,N).
	\]
	to the existence of the dotted arrow.
\end{cor}
\begin{proof}
	This is straightforward.
\end{proof}

\begin{cor}
	Let $X_0\hookrightarrow X$ be a square zero extension of schemes. Suppose that we have a defomation
of perfect complexes $E_0$ and $F_0$ to complexes on $X$. In other words we have exact triangles
\[
M\rightarrow E \rightarrow E_0\quad\mbox{and}\quad N\rightarrow F\rightarrow F_0.
\]	
Then given a diagram 
\begin{center}
	\begin{tikzpicture}
	\matrix (m) [matrix of math nodes,row sep=3em,column sep=4em,minimum width=2em] {
		M& E& E_0 & M[1]\\  
		N& F& F_0 & N[1]\\};
	\path[->] (m-1-1)  edge(m-1-2);
	\path[->] (m-1-2)  edge(m-1-3);
	\path[->] (m-1-3)  edge(m-1-4);
	
	\path[->] (m-2-1)  edge(m-2-2);
	\path[->] (m-2-2)  edge(m-2-3);
	\path[->] (m-2-3)  edge(m-2-4);

	\path[->] (m-1-1)  edge(m-2-1);
	\path[dashed, ->] (m-1-2)  edge(m-2-2);
	\path[->] (m-1-3)  edge(m-2-3);
	\path[->] (m-1-4)  edge(m-2-4);

	\end{tikzpicture}
	
\end{center}
there is an obstruction $o\in\Ext^1(E_0,N)$ whose vanishing is necessary and sufficient for the
existence of the dashed arrow.
\end{cor}

\begin{cor}
	Let $E$ be a perfect complex with locally free cohomology groups. Then the stack
	${\rm Aut}_{X/X}(E)\rightarrow X$ is formally smooth.
\end{cor}
\begin{proof}
	The question is local on $X$ so we may assume that $E$ is strictly perfect and quasi-isomorphic to its cohomology groups. As the automorphism stack is invariant under quasi-isomorphism the result follows from the smoothness of the general linear group.
\end{proof}
\begin{prop}\label{prop:etale Hilbert 90}
If $Aut(E)$ is formally smooth then any form $F$ of $E$ is \'etale locally quasi-isomorphic to $E$.
\end{prop}

\begin{proof}
This follows from \ref{c:smooth} and \cite[Proposition 3.24 (b)]{milne}.
\end{proof}

\section{A Hilbert 90 theorem}\label{s:Zariski Hilbert 90}

\begin{lemma}\label{l:bb}
	Suppose that $P^\bullet$ and $Q^\bullet$ are perfect complexes over a noetherian ring $R$. Let $R\rightarrow S$ be a flat morphism. Then 
	\[
	\Hom_{D^b(S)}(P^\bullet\otimes_R S, Q^\bullet \otimes_R S) \cong
	\Hom_{D^b(R)}(P, Q) \otimes_R S
	\]
\end{lemma}

\begin{proof} Our hypotheses imply that the notions of strictly perfect and perfect agree.
	Hence, we can assume that $P^\bullet$ and $Q^\bullet$ are bounded complexes of projective modules.
	We will write $\Hom^\bullet(P^\bullet, Q^\bullet)$ for the hom complex between these two complexes.
	The $0$th cohomology of this complex computes $\Hom_{D^b(R)}(P, Q)$. Notice that 
	\[
	\Hom^\bullet_S(P^\bullet\otimes_R S, Q^\bullet\otimes_R S)\cong \Hom^\bullet_R(P^\bullet, Q^\bullet)\otimes_R S, \]
	as the analogous formula holds termwise for this complex. The result follows by observing that 
	$\otimes_R S$ commutes with taking cohomology, as $R\rightarrow S$ is flat.
\end{proof}

\begin{cons}
	In the situation of the previous Lemma, take $\bbf = (f_1,\ldots, f_n)\in\Hom_{D^b(R)}(P^\bullet, Q^\bullet)^n$.
	The morphism $R\rightarrow R[t_1, t_2,\ldots , t_n]=R[t_*]$ is flat. We obtain a generic morphism
	\[
	\bbf[t_*]=\sum t_i\otimes f_i : P^\bullet \otimes_R R[t_*] \rightarrow 
	Q^\bullet \otimes_R R[t_*] 
	\]
The morphism $\bbf[t_*]$ fits into an exact triangle
\[
 P^\bullet \otimes_R R[t_*] \rightarrow 
Q^\bullet \otimes_R R[t_*] \rightarrow C(\bbf[t_*])
\]
The complex $C(\bbf[t_*])$ is perfect on $\AA_R^n$ and let $U_\bbf$ be the complement
of the supports of its cohomology groups. It is a possible empty open subsscheme of 
$\AA_R^n$.
\end{cons}

%

\begin{theorem}\label{thm:Zariski Hilbert 90}
	Let $(R, m)$ be a noetherian local ring with infinite residue field $R/m$. Let $R\rightarrow S$
	be a flat morphism. Suppose that $P^\bullet$ and $Q^\bullet$ are strictly perfect complexes over $R$.
	If there is a quasi-isomorphism $P^\bullet \otimes_RS \cong Q^\bullet \otimes_R S$
	then there is a quasi-isomorphism $P^\bullet\cong Q^\bullet $.
\end{theorem}

\begin{proof}
	
	Let $f:P^\bullet\otimes_R S\rightarrow Q^\bullet\otimes_R S$ and 
	$g:Q^\bullet\otimes_R S \rightarrow P^\bullet\otimes_R S$ be quasi-isomorphisms. By Lemma \ref{l:bb} we 
	obtain lists of homomorphisms
	\[
	\bbf = (f_1,\ldots, f_n)\quad\mbox{and}\quad \bbg = (g_1,\ldots,g_m)
	\]
	with the property that
	\[
	g=\sum_i g_i\otimes s_i\qquad f=\sum_if_i\otimes s_i.
	\]
	The open subschemes $U_\bbf$ and $U_\bbg$ of $\AA_R^n$ and $\AA_R^m$ respectively are 
	non-empty as they have an $S$-point. It follows that they have an $R/m$-point which lifts to 
	an $R$-point we call $\alpha$. Taking derived pullbacks of the generic morphisms along the $R$-point gives us  morphisms of complexes 
	\[
	\bar{f}:P^\bullet \rightarrow Q^\bullet\qquad \bar{g}:P^\bullet\rightarrow Q^\bullet.
	\]
	These morphisms are quasi-isomorphisms as they are derived pullbacks of
	quasi-isomorphims.
\end{proof}

\begin{corollary}
	Let $X$ be a scheme and $E$ a perfect complex on $X$.
In the situation of the theorem, let $f:{\rm Spec}(R)\rightarrow B{\Aut}(E)$ be a 
point. Then $f$ factors through the presentation $*\rightarrow B{\rm Aut}(E)$.
\end{corollary}

\begin{proof}
	The morphism $f$ is amounts to giving a morphism $\phi:{\rm Spec}(R)\rightarrow X$ and a fppf-form of $L^*\phi(E)$. The form is trivial via the theorem, that
	is it is quasi-isomorphic to $L^*\phi(E)$. Hence the lift exists.
\end{proof}

\end{document}